\documentclass[1p]{elsarticle}

\usepackage{amsmath}
\usepackage{amsthm}
\usepackage{amssymb}
\usepackage{graphicx}
\usepackage{enumerate}
\usepackage{environ}
\usepackage{mathtools}
\RequirePackage[]{algorithm2e}
\usepackage{calc}
\usepackage{tikz,pgfplots}
\usepackage{todonotes}
\usepackage{multicol}
\usepackage[utf8x]{inputenc}
\usepackage{hyperref}

\newtheorem{theorem}{Theorem}[section]
\newtheorem{proposition}[theorem]{Proposition}

\newtheorem{corollary}[theorem]{Corollary}

\theoremstyle{definition}
\newtheorem{definition}[theorem]{Definition}
\newtheorem{example}[theorem]{Example}
\newtheorem{problem}[theorem]{Problem}

\theoremstyle{remark}
\newtheorem{remark}[theorem]{Remark}

\numberwithin{equation}{section}

\newcommand{\comment}[1]{}

\biboptions{authoryear}

\begin{document}

\begin{frontmatter}

\title{$\lambda$--Core Distance Partitions}

\author{Xandru Mifsud\corref{mycorrespondingauthor}}
\address{Faculty of Science, University of Malta, Msida, Malta}
\ead{xmif0001@um.edu.mt}

\begin{abstract}
The $\lambda$--core vertices of a graph correspond to the non--zero entries of some eigenvector of $\lambda$ for a universal adjacency matrix $\mathbf{U}$ of the graph. We define a partition of the vertex set $V$ based on the $\lambda$--core vertex set and its neighbourhoods at a distance $r$, and give a number of results relating the structure of the graph to this partition. For such partitions, we also define an entropic measure for the information content of a graph, related to every distinct eigenvalue $\lambda$ of $\mathbf{U}$, and discuss its properties and potential applications. 
\end{abstract}

\begin{keyword}
Universal Adjacency Matrices\sep Core Vertices\sep Graph Labelling \sep Information Content
\MSC[2010] 05C50\sep  15A18
\end{keyword}

\end{frontmatter}

\begin{center}
    Published in \textit{Linear Algebra and Its Applications}, \href{url:https://doi.org/10.1016/j.laa.2020.12.012}{https://doi.org/10.1016/j.laa.2020.12.012}
\end{center}

\section{Introduction}

\subsection{Nomenclature}

An undirected graph $G = (V, E)$ has a finite vertex set $V = \{v_1,\dots, v_n\}$ and an edge set $E$ of 2--element subsets of $V$. We use $n$ exclusively to denote the number of vertices. The graphs we consider are connected and without loops or multi--edges. We predominantly use the definitions of {\cite{CvRowSim}}, with notation adapted accordingly. 

The adjacency matrix of a graph $G$ is the $n \times n$ matrix $\mathbf{A} = (a_{ij})$ such that $a_{ij} = 1$ if the vertices $v_i$ and $v_j$ are adjacent and $a_{ij} = 0$ otherwise. The degree matrix is the $n \times n$ matrix $\mathbf{D} = \text{diag}\left(\rho(v_1),\dots, \rho(v_n)\right)$, where $\rho(v)$ is the number of edges incident to $v$. Let $\mathbf{I}$ and $\mathbf{J}$ denote the $n \times n$ identity and all-ones matrices, respectively. {\cite{UnivHaem}} defined the universal adjacency matrix $\mathbf{U} = \gamma_A \mathbf{A} + \gamma_D \mathbf{D} + \gamma_I \mathbf{I} + \gamma_J \mathbf{J}$, for some scalars $\gamma_A \neq 0, \gamma_D, \gamma_I$ and $\gamma_J$. Depending on the choice of parameters, $\mathbf{U}$ encodes a number of well--known and important matrices in graph theory, most notably the adjacency and Laplacian matrices.

The eigenvalue spectrum of a symmetric $n \times n$ matrix $\mathbf{M}$ is a multi--set denoted by ${\rm spec}(\mathbf{M})$, where for $\lambda^{\left(m_{\mathbf{M}}(\lambda)\right)} \in {\rm spec}(\mathbf{M})$, $\lambda$ is the eigenvalue and $m_{\mathbf{M}}(\lambda)$ is the multiplicity of $\lambda$ in the spectrum. For brevity, we shall (with some abuse of notation) simply write $\lambda \in {\rm spec}(\mathbf{M})$.

A vertex $v_i$ of a graph $G$ is a $\lambda$--core vertex for $\mathbf{U}$ if there exists an eigenvector $\mathbf{x}$ of $\mathbf{U}$ such that $\mathbf{Ux} = \lambda\mathbf{x}$ and the $i^{\rm th}$ entry of $\mathbf{x}$ is non--zero. Otherwise, $v_i$ is said to be a $\lambda$--core--forbidden vertex. The sets of $\lambda$--core and $\lambda$--core--forbidden vertices for $\mathbf{U}$ are denoted by $CV_\lambda$ and $CFV_\lambda$ respectively. Diagrammatically, we draw \tikz{\node [fill={rgb,255: red,0; green,172; blue,172}, draw=black, shape=circle, scale=0.6] {v}} to represent $\lambda$--core vertices, \tikz{\node [fill=white, draw=black, shape=circle, scale=0.6] {u}} for vertices which have a $\lambda$--core neighbour but are not $\lambda$--core themselves, and \tikz{\node [fill={rgb,255: red,183; green,183; blue,183}, draw=black, shape=circle, scale=0.6] {w}} otherwise.

The terms \textit{core} and \textit{core--forbidden} vertices were first introduced by {\cite{SciCVDefn}} for the kernel of $\mathbf{A}$. Similar concepts exist: for example, in linear algebra the core--forbidden vertices that increase the multiplicity of $\lambda$ are termed as \textit{Parter} vertices (after the work of S. Parter during the 1960s), while those that leave the multiplicity unchanged are called \textit{neutral} vertices. The core vertices are termed as \textit{downer} vertices, since they decrease the multiplicity of $\lambda$. In the context of the Laplacian, the core--forbidden vertices are typically known as the \textit{Fiedler} vertices, after the pioneering work by M. Fiedler throughout the 1970s and 1980s.

More recently, new terminology has been introduced for different contexts. Of significant importance are the $0$--core and $0$--core--forbidden vertices, which have been studied greatly for singular trees. The \textit{nullity} $\eta(G)$ of a graph is the dimension of the kernel of $\mathbf{A}$. In \cite{FerCruz}, the $0$--core--forbidden vertices of a tree are termed as $F$--vertices. Those vertices that, in particular, increase the nullity of $T$ or the multiplicity of $0$ for the Laplacian of $T$, are termed as $P$--vertices.

\subsection{Overview}

For a graph $G$ with independent $\lambda$-core vertices, the set of their neighbours $N(CV_\lambda)$ is disjoint from $CV_\lambda$. Hence one can define a partition of $V$ with blocks $CV_\lambda$, $N(CV_\lambda)$, and $CFV_{R, \lambda}$ (which is the set of remaining vertices not in $CV_\lambda$ and $N(CV_\lambda)$). Such a partitioning was studied closely for the eigenvalue $\lambda = 0$ of $\mathbf{A}$ by {\cite{ScMfBg}} for graphs with independent $CV_0$, and by {\cite{JauMol}} for the case of singular trees (which always have an independent $CV_0$). In Section 2, we introduce the $\lambda$--Core Distance Partition ($\lambda$--CDP), which is a natural generalisation for the case when $CV_\lambda$ is not an independent set and hence $N(CV_\lambda)$ is not disjoint from $CV_\lambda$. Moreover, we study the relation between the $\lambda$--CDP and the symmetries of the graph, and how the core distance partitions of two distinct eigenvalues relate to each other. 

In Section 3 we define an index for the information content of a graph related to its $\lambda$--CDP, and study its lower and upper bounds for a fixed number of vertices $n$ and varying number of $\lambda$--core vertices, $|CV_\lambda|$. Lastly in Section 4 we consider a number of applications related to this index and the structure of graphs with a singular adjacency matrix, most notably molecular graphs.

\section{$\lambda$--Core Distance Partitions}

Consider a connected graph $G$, and $\lambda \in \text{spec}\left(\mathbf{U}\right)$. The function $d(u, v) \colon V \times V \rightarrow \mathbb{N} \ \dot\cup \ \{0\}$ gives the length of the shortest path between any two vertices $u$ and $v$. We define the minimum distance from any vertex $v \in V$ to some $\lambda$--core vertex $u \in CV_\lambda$ by the function $d_\lambda (v) \colon V \rightarrow \mathbb{N} \ \dot\cup \ \{0\}$,
\begin{equation} \label{dYeqn}
    d_\lambda(v) = \min\left\{d(v, u) \colon u \in CV_\lambda\right\}.
\end{equation}
The non--negative integer $D_\lambda = \max\left\{d_\lambda(v) \colon v \in V\right\}$ is the maximum distance assigned by $d_\lambda$ to a vertex, which is bounded above by the diameter of the graph. We define the set of vertices at a minimum distance $0 \leq i \leq D_\lambda$ from $\lambda$--core vertex as the set $V_{\lambda, i} = \left\{v \in V \colon d_\lambda (v) = i\right\}$. It follows then that $CV_\lambda = V_{\lambda, 0}$. 

\begin{proposition}[$\lambda$--Core Distance Partition] \label{distParn}
Let $\lambda \in {\rm spec}(\mathbf{U})$. The collection $\mathcal{V}_\lambda = \left\{V_{\lambda, 0},\dots, V_{\lambda, D_\lambda}\right\}$ is a partition of $V$.
\end{proposition}

\begin{proof}
    Since $G$ is connected, there exists a path between any pair of vertices. In particular, every vertex $v \in V$ has a shortest path to some $\lambda$--core vertex. Hence $d_\lambda$ exists for every vertex $v \in V$ and therefore $v$ belongs to some $V_{\lambda, i} \in \mathcal{V}_\lambda$, for $0 \leq i \leq D_\lambda$. Clearly then,  $V$ is the union of $V_{\lambda, 0},\dots, V_{\lambda, D_\lambda}$.
    
    Moreover, since $d_\lambda$ is clearly not one--to--many, we have that the intersection of $V_{\lambda, i}$ and $V_{\lambda, j}$ is empty, for $i \neq j$. The result follows.
\end{proof}

\begin{remark} \label{eqrelnRem}
Note that $d_\lambda$ gives an equivalence relation, where two vertices $u$ and $v$ are said to be related if $d_\lambda (u) = d_\lambda (v)$. Hence the $\lambda$--CDP $\mathcal{V}_\lambda$ is in fact a collection of equivalence classes.
\end{remark}

\subsection{$\lambda$--Core Distance Partitions and Orbit Structure}

The automorphism group $\text{Aut}(G)$ of a graph is the group from $G$ onto itself. If $v_i$ is adjacent to $v_j$ then for any automorphism $\sigma \in \text{Aut}(G)$, $\sigma(v_i)$ is adjacent to $\sigma(v_j)$. The \textit{orbit} of $\text{Aut}(G)$ on $v \in V$ is the set $\pi(v) = \left\{\sigma (v) \colon \sigma \in \text{Aut}(G)\right\}$. We denote the set of orbits by $\Pi$. A graph $G$ is vertex--transitive if it has a single orbit. The orbits of $\text{Aut}(G)$ induce a partition $\Pi = \left\{\pi_1, \pi_2,\dots, \pi_s\right\}$ of $V$, known as the \textit{orbital partition} of $G$.

\begin{figure}[ht!]
    \centering
    \begin{tikzpicture}
    \tikzstyle{ncv}=[thick, fill=white, draw=black, shape=circle]
    \tikzstyle{cv}=[thick, fill={rgb,255: red,0; green,172; blue,172}, draw=black, shape=circle]
    \tikzstyle{cfvR}=[thick, fill={rgb,255: red,183; green,183; blue,183}, draw=black, shape=circle]
    
    \tikzstyle{line}=[-, draw=black]
    \tikzstyle{edge}=[thick, -, draw=black]
    \tikzstyle{newEdge}=[-, draw={rgb,255: red,0; green,8; blue,160}]
    \tikzstyle{delEdge}=[-, draw={rgb,255: red,187; green,7; blue,34}]
    \tikzstyle{dirEdge}=[thick, ->]

		\node [font=\small, style=cv] (0) at (-3.5, 1.5) {1};
		\node [font=\small, style=cv] (1) at (-3.5, 0.5) {2};
		\node [font=\small, style=cv] (2) at (-3.5, -0.5) {3};
		\node [font=\small, style=cv] (3) at (-2, 2) {4};
		\node [font=\small, style=cv] (4) at (-2, 1) {5};
		\node [font=\small, style=cv] (5) at (-2, 0) {6};
		\node [font=\small, style=ncv] (6) at (-0.5, 1.5) {7};
		\node [font=\small, style=ncv] (7) at (-0.5, 0.5) {8};
		\node [font=\small, style=ncv] (8) at (-0.5, -0.5) {9};
		\node [font=\small, style=cfvR, scale=0.85] (9) at (2, 1.5) {10};
		\node [font=\small, style=cfvR, scale=0.85] (10) at (1, 0.5) {11};
		\node [font=\small, style=cfvR, scale=0.85] (11) at (2, -0.5) {12};
		\node [font=\small] (12) at (-4, -1) {};
		\node [font=\small] (13) at (-4, -1.5) {};
		\node [font=\small] (14) at (-4, -1.25) {};
		\node [font=\small] (15) at (-1.5, -1) {};
		\node [font=\small] (16) at (-1.5, -1.25) {};
		\node [font=\small] (17) at (-1.5, -1.5) {};
		\node [font=\small] (18) at (0.5, -1.25) {};
		\node [font=\small] (19) at (0.5, -1) {};
		\node [font=\small] (20) at (0.5, -1.5) {};
		\node [font=\small] (21) at (2.5, -1.25) {};
		\node [font=\small] (22) at (2.5, -1) {};
		\node [font=\small] (23) at (2.5, -1.5) {};
		\node [font=\small] (24) at (-2.75, -1.5) {$V_{1,0}$};
		\node [font=\small] (25) at (-0.5, -1.5) {$V_{1,1}$};
		\node [font=\small] (26) at (1.5, -1.5) {$V_{1,2}$};
		\node [font=\small] (27) at (-4.5, 0.5) {G};

		\draw (0) to (3);
		\draw (0) to (1);
		\draw (1) to (2);
		\draw (2) to (5);
		\draw (5) to (4);
		\draw (4) to (3);
		\draw (0) to (6);
		\draw (3) to (6);
		\draw (1) to (7);
		\draw (4) to (7);
		\draw (2) to (8);
		\draw (5) to (8);
		\draw (6) to (9);
		\draw (7) to (10);
		\draw (8) to (11);
		\draw (9) to (11);
		\draw (11) to (10);
		\draw (10) to (9);
		\draw [style=edge] (12.center) to (13.center);
		\draw [style=edge] (15.center) to (17.center);
		\draw [style=edge, in=180, out=0] (14.center) to (16.center);
		\draw [style=edge] (16.center) to (18.center);
		\draw [style=edge] (19.center) to (20.center);
		\draw [style=edge] (22.center) to (23.center);
		\draw [style=edge] (18.center) to (21.center);
\end{tikzpicture}
    \caption{$1$--CDP of the graph $G$, for $\mathbf{U} = \mathbf{A}$.}
    \label{fig:cubicExample}
\end{figure}
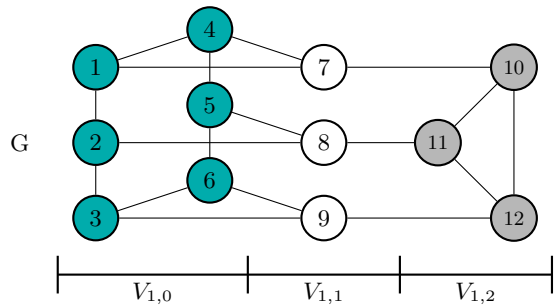

Consider the cubic graph $G$ in Figure \ref{fig:cubicExample} as a motivating example. The spectrum of its adjacency matrix (up to 3 decimal places) is: $\left\{3^{(1)}, 2.115^{(1)}, 2^{(1)}, 1.303^{(1)}, 1^{(1)}, -0.254^{(1)}, -1^{(3)}, -1.861^{(1)}, -2^{(1)}, -2.303^{(1)}\right\}.$ For $\lambda = 1$, the $1$--CDP for $\mathbf{U} = \mathbf{A}$ is: $$\mathcal{V}_1 \ = \left\{V_{1, 0} = \{1, 2, 3, 4, 5, 6\}, \ V_{1,1} = \{7, 8, 9\}, \ V_{1,2} = \{10, 11, 12\}\right\}$$ Moreover, the orbits of ${\rm Aut}(G)$ acting on $V$ are, $$\Pi = \left\{\pi_1 = \{1, 3, 4, 6\}, \pi_2 = \{2, 5\}, \pi_3 = \{7, 9\}, \pi_4 = \{8\}, \pi_5 = \{10, 12\}, \pi_6 = \{11\}\right\}$$ Notice that each block in the $1$--CDP is a disjoint union of orbits in $\Pi$. In this section we will show that this holds for every $\lambda$--CDP of any universal adjacency matrix $\mathbf{U}$ of $G$.

Given a permutation $\sigma \in \text{Aut}(G)$, there exists a corresponding permutation matrix $\mathbf{P_\sigma}$ such that $\mathbf{P_\sigma A} = \mathbf{AP_\sigma}$. Since $\mathbf{D}$ and $\mathbf{I}$ have constant entries with respect to each orbit, and since $\mathbf{J}$ is a constant matrix, then $\mathbf{P}_\sigma$ also commutes with $\mathbf{D}$, $\mathbf{I}$, and $\mathbf{J}$. Since matrix multiplication is distributive, then $\mathbf{U}$ and $\mathbf{P}_\sigma$ commute: $\mathbf{U}\mathbf{P}_\sigma = \mathbf{P}_\sigma \mathbf{U}$. The following result is due to {\cite{Akbari}}.

\begin{proposition}
[{\cite{Akbari}}] \label{VertTransitiveCore}
Let $G$ be vertex--transitive and $\lambda \in {\rm spec}(\mathbf{A})$. Then every vertex of $G$ is a $\lambda$--core vertex.
\end{proposition}

The following proposition is a generalisation of Proposition \ref{VertTransitiveCore}, of which it is a special case for when $G$ is vertex--transitive and hence has a single orbit.

\begin{proposition} \label{orbsCoreVert}
Let $\lambda \in {\rm spec}(\mathbf{U})$ and let $u, v$ belong to the same orbit $\pi \in \Pi$. If $u$ is a $\lambda$--core vertex, then $v$ is also a $\lambda$--core vertex. 
\end{proposition}

\begin{proof}
Since $u$ and $v$ belong to the same orbit, it follows by the definition of an orbit that there exists a permutation $\sigma \in \text{Aut}(G)$, such that $\sigma(u) = v$. Let $\mathbf{P_\sigma}$ be the corresponding permutation matrix. Since $u$ is $\lambda$--core, there exists an eigenvector $\mathbf{x}$ for $\lambda$ such that the $u^{\rm th}$ entry of $\mathbf{x}$ is non--zero. Now, $\mathbf{UP_\sigma x} = \mathbf{P_\sigma Ux} = \lambda \mathbf{P_\sigma x}$, since $\mathbf{UP_\sigma} = \mathbf{P_\sigma U}$ and $\mathbf{Ux} = \lambda\mathbf{x}$. Hence $\mathbf{P_\sigma x}$ is an eigenvector of $\mathbf{U}$ as well. Since $\mathbf{P_\sigma}$ permutes the $u^{\text{th}}$ entry of $\mathbf{x}$ with the $v^{\text{th}}$ entry, it follows that the $v^{\rm th}$ entry of $\mathbf{P_\sigma x}$ is non--zero. Therefore $v$ is a $\lambda$--core vertex.
\end{proof}

\begin{remark} \label{cv_orbs}
It follows that all the vertices in the same orbit are either $\lambda$--core or $\lambda$--core forbidden. Hence, $CV_\lambda$ is the disjoint union of some collection of orbits $\left\{\pi_{i_1},\dots, \pi_{i_s}\right\} \subseteq \Pi$.
\end{remark}

\begin{theorem} \label{decompThm}
Let $\lambda \in {\rm spec}(\mathbf{U})$ and $0 \leq i \leq D_\lambda$. Then $V_{\lambda, i}$ is the disjoint union of a collection of orbits $\left\{\pi_1,\dots, \pi_{s_i} \right\} \subseteq \Pi$.
\end{theorem}

\begin{proof}
That $V_{\lambda, 0} = CV_\lambda$ is the disjoint union of orbits is discussed in Remark \ref{cv_orbs}. Suppose that up to some $0 \leq i < D_\lambda$, $\forall 0 \leq k \leq i, V_{\lambda, k}$ is the disjoint union of some collection of orbits $\Pi_k = \left\{\pi^{(k)}_1,\dots, \pi^{(k)}_{s_k} \right\} \subseteq \Pi$. We show that the claim holds for $i+1$.

Suppose that $V_{\lambda, i+1}$ is not the disjoint union of some collection of orbits. Then there exists two vertices $u$ and $v$ in an orbit $\pi$, such that $v$ is in $V_{\lambda, i+1}$ and $u$ is in $V_{\lambda, j}$, for some $i+2 \leq j \leq D_\lambda$. Note that $\pi$ is distinct from those in the collection $\dot\cup_{k = 0}^i \Pi_k$.

By definition of the $\lambda$--CDP, $v$ in $V_{\lambda, i+1}$ must have a neighbour in $V_{\lambda, i}$. Moreover, there must exist a permutation $\sigma \in {\rm Aut}(G)$ which maps $v$ to $u$ while preserving adjacency. In particular, $V_{\lambda, k}$ for $0 \leq k \leq i$ are preserved since they are the disjoint union of orbits, and hence $u$ must have a neighbour in $V_{\lambda, i}$. Since $i + 1 < j$, then $u$ has a path to some $\lambda$--core vertex of length less than $j$, and hence cannot be in $V_{\lambda, j}$, a contradiction. Therefore $V_{\lambda, i+1}$ is the disjoint union of some collection of orbits, completing the proof.
\end{proof}

\subsection{Relating the Core Distance Partitions of Two Distinct Eigenvalues}

We continue our study of the relation between the spectrum of $\mathbf{U}$ and the structure of $G$, in light of the $\lambda$--CDP. We say that for two distinct eigenvalues $\lambda$ and $\mu$, the $\mu$--CDP can be \textit{constructed} from the $\lambda$--CDP if every block in the $\mu$--CDP is a disjoint union of blocks in the $\lambda$--CDP. In particular, we are interested in the following problem.

\begin{problem} \label{constProb}
Let $\lambda \in {\rm spec}(\mathbf{U})$ and let the $\lambda$--CDP be given. When does there exist an eigenvalue $\mu \neq \lambda$ of $\mathbf{U}$, such that the $\mu$--CDP can be constructed from the $\lambda$--CDP?
\end{problem}

As it turns out, if the $\lambda$--CDP is an equitable partition, then under certain circumstances reconstruction is possible. First however we require some preliminary definitions and results. The \textit{characteristic matrix} $\mathbf{C}$ of a partition $\{V_1,\dots, V_k\}$ is an $n \times k$ matrix, where the $i^{\rm th}$ column $\mathbf{c}_i$ is associated with the block $V_i$, such that for every vertex in $V_i$ the corresponding entry in $\mathbf{c}_i$ is 1, and 0 otherwise. An \textit{equitable partition} with $k$ blocks is a partition of $V$ such that there exists a $k \times k$ matrix $\mathbf{B}$, known as the \textit{divisor matrix}, satisfying $\mathbf{U}\mathbf{C} = \mathbf{C}\mathbf{B}$.

\begin{remark} \label{eqVectRemark}
Let $\mathbf{x}$ be an eigenvector for $\mu$ of $\mathbf{B}$ such that its $i^{\rm th}$ entry $x_i$ is non--zero. If $\mathbf{C}\mathbf{x}$ is also an eigenvector for some $\lambda \in {\rm spec}(\mathbf{U})$, then all the vertices in $V_i$ must be $\lambda$--core vertices, since the corresponding entry in $\mathbf{C}\mathbf{x}$ is $x_i$ by definition of $\mathbf{C}$.
\end{remark}

It is well--known that for the case when $\mathbf{U} = \mathbf{A}$, the characteristic polynomial of $\mathbf{B}$ divides the characteristic polynomial of $\mathbf{A}$ (Theorem 3.9.5, {\cite{CvRowSim}}). More generally, the result holds for any universal adjacency matrix $\mathbf{U}$, which we state as follows without proof.
\begin{theorem} \label{charPolynDiv}
    Let $\mathcal{C}$ be an equitable partition of $V$ with divisor and characteristic matrices $\mathbf{B}$ and $\mathbf{C}$, respectively. The characteristic polynomial of $\mathbf{B}$ divides the characteristic polynomial of $\mathbf{U}$.
\end{theorem}

\begin{remark} \label{eqCVThmRemark}
Consider some $\lambda$--CDP $\mathcal{V}_\lambda$ to be equitable. By combining Remark \ref{eqVectRemark} and Theorem \ref{charPolynDiv} we observe that for every eigenvalue of $\mathbf{U}$ that is also an eigenvalue of $\mathbf{B}$, the core vertex set has a subset which is the disjoint union of some blocks in $\mathcal{V}_\lambda$.
\end{remark}

\begin{theorem} \label{eqCVThm}
    Let $\lambda \in {\rm spec}(\mathbf{U})$ and $\mathcal{V}_\lambda$ be an equitable partition with divisor and characteristic matrices $\mathbf{B}$ and $\mathbf{C}$, respectively. Let $\mu \in {\rm spec}(\mathbf{B})$. Then there exists a subset of $CV_\mu$ which is the disjoint union of a finite collection of blocks in $\mathcal{V}_\lambda$. Moreover, if $m_{\mathbf{U}}(\mu) = m_{\mathbf{B}}(\mu)$ then equality holds between the two sets and $\mathcal{V}_\mu$ can be constructed from $\mathcal{V}_\lambda$.
\end{theorem}

\begin{proof}
Let $B$ be a basis for the eigenspace of $\mu$ for $\mathbf{B}$. Define the indexing set $X_B$ of row--indices corresponding to some non--zero row of any vector $\mathbf{b}$ in $B$: $X_B = \left\{i \in \left\{1,\dots, D_\lambda + 1\right\} : \exists \mathbf{b} \in B , \mathbf{b}\cdot\mathbf{e}_i \neq 0\right\}$.
By definition of an equitable partition, each index $i$ in $X_B$ corresponds to the block $V_{\lambda, i-1}$ of $\mathcal{V}_\lambda$. As discussed in Remark \ref{eqCVThmRemark} we have that
$\underset{i \in X_B}{{\dot\cup}} V_{\lambda, i-1} \subseteq CV_\mu$
and the first part of the result follows.

Let the basis $B$ be $\{\mathbf{b}_1,\dots, \mathbf{b}_r\}$ and consider the set $B^\prime = \{\mathbf{C}\mathbf{b}_1,\dots, \mathbf{C}\mathbf{b}_r\}$. Since $B$ is an independent set, then $B^\prime$ must also be an independent set. By Remark \ref{eqVectRemark} and Theorem \ref{charPolynDiv}, $B^\prime$ is a basis for some sub--space of the eigenspace of $\mu$ for $\mathbf{U}$. In the case that $m_{\mathbf{U}}(\mu) = m_{\mathbf{B}}(\mu)$, it follows that $B^\prime$ is a basis for the eigenspace of $\mu$ for $\mathbf{U}$. Since every vertex in $CV_\mu$ must correspond to a non--zero entry for some eigenvector in $B^\prime$, it follows that $\underset{i \in X_B}{{\dot\cup}} V_{\lambda, i-1} = CV_\mu$. Let $X_B = X^0_B$. For $d \in \mathbb{N}$, recursively define,
\vspace{-3mm}
\begin{equation*}
X^d_B = \left\{i \in \{1,\dots, D_\lambda + 1\} : \exists j \in X^{d-1}_B, i = j \pm 1\right\} \backslash \bigcup\limits_{k=0}^{d-1} X^k_B
\end{equation*}
which is the set of indices, not already appearing in some $X^j_B$ for $0 \leq j < d$, of blocks whose vertices are at a distance 1 from the vertices in the blocks with indices in $X^{d-1}_B$. Hence the union of the blocks $V_{\lambda, i-1}$ with indices in $X^{d}_B$ gives the vertices at a distance $d$ from $CV_\mu$. Moreover, $\exists D_\mu \in \mathbb{N}$ such that $X_B^{D_\mu} \neq \emptyset$ and for every $d > D_\mu$, $X_B^d = \emptyset$. Hence $\mathcal{V}_\mu$ can be constructed from the indexing sets $X^0_B,\dots, X^{D_\mu}_B$.
\end{proof}

\subsubsection{Worked Example} \label{workedExample}

Consider once again the cubic graph $G$ in Figure \ref{fig:cubicExample}. The $1$--CDP is also an equitable partition, having the following divisor and characteristic matrices $\mathbf{B}$ and $\mathbf{C}$, respectively, 
\begin{eqnarray*}
\mathbf{B} = \left(\begin{array}{ccc}
     2 & 1 & 0 \\
     2 & 0 & 1 \\
     0 & 1 & 2
\end{array}\right)
&
\mathbf{C} = \left(
\begin{array}{cccccccccccc}
 1 & 1 & 1 & 1 & 1 & 1 & 0 & 0 & 0 & 0 & 0 & 0 \\
 0 & 0 & 0 & 0 & 0 & 0 & 1 & 1 & 1 & 0 & 0 & 0 \\
 0 & 0 & 0 & 0 & 0 & 0 & 0 & 0 & 0 & 1 & 1 & 1 \\
\end{array}
\right)^\intercal.
\end{eqnarray*}
The spectrum of $\mathbf{B}$ is $\{3^{(1)}, 2^{(1)}, -1^{(1)}\}$. Consider $\mu = 2$, which is in both ${\rm spec}(\mathbf{A})$ and ${\rm spec}(\mathbf{B})$, with the same multiplicity (equal to 1). The basis of the eigenspace for $2$ of $\mathbf{B}$ is $\left\{(-1, 0, 2)^\intercal\right\}$, which is non--zero with respect to each block of $\mathcal{V}_1$. On the other hand, the eigenspace for $2$ of $\mathbf{A}$ is spanned by a single vector, $\left(-1, -1, -1, -1, -1, -1, 0, 0, 0, 2, 2, 2\right)^\intercal$. In this case, $CV_2$ is the disjoint union of the blocks $V_{1, 0}$ and $V_{1, 2}$ from $\mathcal{V}_1$. Moreover, $\mathcal{V}_2$ can be constructed from $\mathcal{V}_1$:  $\mathcal{V}_2 \ = \left\{V_{2, 0} = V_{1,0} \ \dot\cup \ V_{1, 2}, \ V_{2, 1} = V_{1, 1}\right\}$.

\section{Information Content of Graphs and Topological Indices}

\subsection{Information Content and Structure of $\lambda$--Core Distance Partitions}

The study of the information content of graphs is an important and active field, which is at the intersection of graph and information theory. The notion of information content of graphs was first introduced by \cite{Rashevsky, Trucco}, and was generalised and formalised by \cite{Mowshowitz} as follows: Given a graph $G$ with a vertex set $V$ and a partition of $V$ into equivalence classes, a finite probability scheme can be assigned to this partitioning, by assigning the probability $\frac{\left|V_i\right|}{n}$ to each equivalence class $V_i$. An entropic measure associated with this finite probability scheme could then be defined as $I = -\sum_{i = 0}^k\frac{\left|V_i\right|}{n}\log\left(\frac{\left|V_i\right|}{n}\right)$, which is based on Shannon's Entropy. Such measures form part of a broader class of indices, known as \textit{topological indices}. Note that the logarithms we consider are typically of base 2, $e$ or 10, all of which are monotonically increasing.

Since the orbital partition $\Pi$ is in fact a collection of equivalence classes, Mowshowitz studied the entropy measure $I_a$ associated with $\Pi$: $I_a = -\sum_{i = 0}^s\frac{\left|\pi_i\right|}{n}\log\left(\frac{\left|\pi_i\right|}{n}\right)$. Similarly, for $\lambda \in {\rm spec}(\mathbf{U})$, the $\lambda$--CDP $\mathcal{V}_\lambda$ is a collection of equivalence classes, as discussed in Remark \ref{eqrelnRem}, and hence a finite probability scheme may be associated with it. In this manner, we can define a topological index for every distinct eigenvalue of $\mathbf{U}$, related to the structure of the neighbourhoods $V_{\lambda, i}$ of $CV_\lambda$ at a distance $i$, and their size.

\begin{theorem}
    Let $\lambda \in {\rm spec}(\mathbf{U})$ and $\mathcal{V}_\lambda$ be the associated $\lambda$--CDP. Then for each equivalence class $V_{\lambda, i}$ in  $\mathcal{V}_\lambda$, we can associate the probability $\frac{\left|V_{\lambda, i}\right|}{n}$, defining a finite probability scheme with entropy $$I_\lambda = -\sum\limits_{i = 0}^{D_\lambda}\dfrac{\left|V_{\lambda, i}\right|}{n}\log\left(\dfrac{\left|V_{\lambda, i}\right|}{n}\right).$$
\end{theorem}

\subsection{Relating $I_\lambda$ and $I_\mu$ for Two Distinct Eigenvalues}

Consider a value $x_i$ associated with each block $V_{\lambda, i}$ in some $\lambda$--CDP. Then, $W_\lambda (x_0,\dots, x_{D_\lambda}) = \frac{1}{n}\sum_{i = 0}^{D_\lambda} \left|V_{\lambda, i}\right| x_i$
is the arithmetic mean of these values, weighted by the size of each block. For brevity, we write $W_\lambda\left(\{x_i\}\right)$, where $0 \leq i \leq D_\lambda$. Let $\mathcal{P}$ be a partition of $V$ such that each block $V_{\lambda, i}$ in $\mathcal{V}_\lambda$ is the disjoint union of $k_i \in \mathbb{N}$ blocks in $\mathcal{P}$.

It can be shown that $-\sum_{i = 1}^{|\mathcal{P}|} \frac{|P_i|}{n} \log\left(\frac{|P_i|}{n}\right)$ is bounded above by $I_\lambda + W_\lambda\left(\{\log(k_i)\}\right)$. We first require the following well--known inequality in information theory.

\begin{theorem}[Theorem 2.7.1, \cite{CoThom}] \label{logsum}
    For non--negative numbers $x_1, \dots, x_n$ and $y_1, \dots, y_n$, $$\sum_{k=1}^n x_k \log\left(\dfrac{x_k}{y_k}\right) \geq \left(\sum_{k=1}^n x_k\right)\log\left(\dfrac{\sum_{k=1}^n x_k}{\sum_{k=1}^n y_k}\right)$$ with equality if and only if $\frac{x_k}{y_k} = {\rm constant}$.
\end{theorem}

\begin{proposition}
    Let $\lambda \in {\rm spec}(\mathbf{U})$ and $\mathcal{V}_\lambda$ be the associated $\lambda$--CDP. Let $\mathcal{P}$ be a partition of $V$ such that each block $V_{\lambda, i}$ of $\mathcal{V}_\lambda$ is the disjoint union of $k_i \in \mathbb{N}$ blocks of $\mathcal{P}$. Then, 
    \begin{equation}
        -\sum_{i = 1}^{|\mathcal{P}|} \frac{|P_i|}{n} \log\left(\frac{|P_i|}{n}\right) \leq I_\lambda + W_\lambda\left(\{\log(k_i)\}\right)
    \end{equation}
\end{proposition}

\begin{proof}
    Let $V_{\lambda, i}$ be the disjoint union of $k_i \in \mathbb{N}$ blocks $\left\{P^{(i)}_1,\dots, P^{(i)}_{k_i}\right\}$ in $\mathcal{P}$. Then,
    \begin{align*}
        &\phantom{{}=1} \dfrac{1}{n}\sum_{k = 1}^{k_i} \left|P_k^{(i)}\right|\log\left(\tfrac{\left|P_k^{(i)}\right|}{n}\right) && \text{since} \ V_{\lambda, i} \ \text{is the union of} \left\{P^{(i)}_1,\dots, P^{(i)}_{k_i}\right\}, \\
        &\geq \dfrac{1}{n}\left(\sum_{k = 1}^{k_i} \left|P_k^{(i)}\right|\right) \log\left(\tfrac{\sum_{k = 1}^{k_i} \left|P_k^{(i)}\right|}{n k_i}\right) && \text{by Theorem \ref{logsum} for} \ x_k = \left|P_k^{(i)}\right|, y_k = n, \\
        &= \tfrac{|V_{\lambda, i}|}{n}\left[\log\left(\tfrac{|V_{\lambda, i}|}{n}\right) - \log(k_i)\right] && \text{by properties of logs and} \ |V_{\lambda, i}|=\!\sum_{k = 1}^{k_i} \left|P_k^{(i)}\right|.
    \end{align*}
    Applying this inequality to each term in $I_\lambda$, the bound follows immediately.
\end{proof}

Such an inequality is closely associated with Problem \ref{constProb}. For two distinct eigenvalues $\lambda, \mu \in {\rm spec}(\mathbf{U})$, such that the $\mu$--CDP can be constructed from the $\lambda$--CDP, the following result follows immediately. 

\begin{proposition}
    Let $\lambda, \mu \in {\rm spec}(\mathbf{U})$ be two distinct eigenvalues. If the $\mu$--CDP can be constructed from the $\lambda$--CDP, then there exist $|D_\mu| + 1$ positive integers $k_i$ such that $I_\lambda \leq I_\mu + W_\mu\left(\{\log(k_i)\}\right)$.
\end{proposition}

In particular, as an immediate consequence of Theorem \ref{eqCVThm}, we have the following corollary. 

\begin{corollary}
    Let $\lambda \in {\rm spec}(\mathbf{U})$ and $\mathcal{V}_\lambda$ be an equitable partition with divisor and characteristic matrices $\mathbf{B}$ and $\mathbf{C}$, respectively. Let $\mu \in {\rm spec}(\mathbf{B})$ such that $m_{\mathbf{U}}(\mu) = m_{\mathbf{B}}(\mu)$. Then $I_\lambda \leq I_\mu + W_\mu\left(\{\log(k_i)\}\right)$.
\end{corollary}

Consider the $1$--CDP and $2$--CDP for the adjacency matrix $\mathbf{A}$ of the graph $G$ in Figure \ref{fig:cubicExample}. As in the worked example of Section \ref{workedExample}, the $2$--CDP can be constructed from the $1$--CDP. In particular, the $2$--CDP has two blocks $V_{2, 0}$ and $V_{2, 1}$, which are the disjoint union of $k_0 = 2$ and $k_1 = 1$ blocks from the $1$--CDP, respectively. Up to 4 decimal places, we have that $W_2\left(\ln(2), \ln(1)\right) \simeq 0.5199$ and $I_2 \simeq 0.5623$. Moreover, $I_1 \simeq 1.0397$ and hence $I_1 \leq I_2 + W_2\left(\ln(2), \ln(1)\right) \simeq 1.0822$. 

\subsection{Note on the Minimum--Maximum Attainable Bounds of $I_\lambda$}

Of interest are the minimum and maximum bounds attainable by $I_\lambda$ for a fixed number of vertices $n$, and varying number of $\lambda$--core vertices. These bounds are given in the following theorem, and visualised in Figure \ref{fig:entropyPlt}.

\begin{proposition} \label{minmaxBdds}
    Let $\lambda \in {\rm spec}(\mathbf{U})$ and let $|CV_\lambda| = k, 1 \leq k \leq n$. Then,
    $${\left.\begin{array}{cr}
        0 & k = n \\
        -\tfrac{n-k}{n}\log\left(\tfrac{n-k}{n}\right) & k < n
    \end{array}\right\}}
    \leq I_\lambda + \dfrac{k}{n}\log\left(\dfrac{k}{n}\right) \leq -\dfrac{n-k}{n}\log\left(\dfrac{1}{n}\right)$$
\end{proposition}

\begin{proof}
    Let $|CV_\lambda| = k$, where $1 \leq k < n$. Then $1 \leq D_\lambda \leq n-k$. Consider,
    \begin{equation*}
        I_\lambda + \tfrac{k}{n} \log\left(\tfrac{k}{n}\right) = -\sum_{i=1}^{D_\lambda} \sum_{j=1}^{|V_{\lambda, i}|} \tfrac{1}{n}\log\left(\tfrac{|V_{\lambda, i}|}{n}\right) \leq -\sum_{i=1}^{D_\lambda} \sum_{j=1}^{|V_{\lambda, i}|} \tfrac{1}{n}\log\left(\tfrac{1}{n}\right)
    \end{equation*}
    since $-\log(x)$ is monotonically decreasing and $1 \leq |V_{\lambda, i}|$. Since $n-k = \sum_{i=1}^{D_\lambda} |V_{\lambda, i}|$, the upper--bound $I_\lambda + \tfrac{k}{n} \log\left(\tfrac{k}{n}\right) \leq \tfrac{n-k}{n}\log\left(\tfrac{1}{n}\right)$ follows. For $1 \leq k < n$, by a similar argument we have $-\tfrac{n-k}{n}\log\left(\tfrac{n-k}{n}\right) \leq I_\lambda + \tfrac{k}{n}\log\left(\tfrac{k}{n}\right)$. For the case that $k = n$, then $I_\lambda = \tfrac{n}{n}\log\left(\tfrac{n}{n}\right) = 0$, completing the proof.
\end{proof}

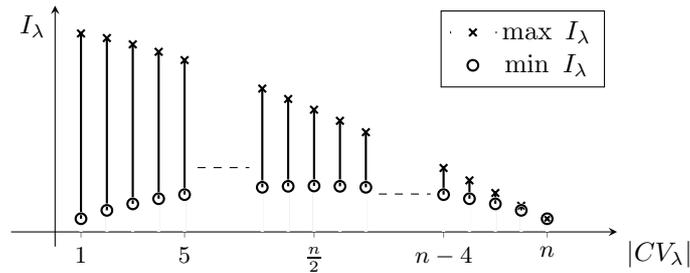
\begin{figure}[ht!]
    \centering
     \pgfplotsset{
    standard/.style={
        axis x line=middle,
        axis y line=middle,
        enlarge x limits=0.15,
        enlarge y limits=0.15,
        every axis x label/.style={at={(current axis.right of origin)},anchor=north west},
        every axis y label/.style={at={(current axis.above origin)},anchor=north east},
        every axis plot post/.style={mark options={fill=white}},
        every tick label/.append style={font=\small}
        }
    }

\begin{tikzpicture}
    \begin{axis}[%
            standard,
            samples = 5,
            xlabel={$|CV_\lambda|$},
            ylabel={$I_\lambda$},
            ytick = \empty,
            xtick = {1,5,10,15,19},
            xticklabels = {$1$,$5$,$\frac{n}{2}$,$n-4$,$n$},
            width=0.6\textwidth,
            height=0.3\textwidth]
                
            \addplot[mark=x,ycomb,black,thick,domain=1:5] {-(x/20)*log2(x/20)-((20-x)/20)*log2((1/20))};
            \addlegendentry{$\max \ I_\lambda$};
            \addplot[mark=o,ycomb,white,very thick,domain=1:5,forget plot] {-(x/20)*log2(x/20)-((20-x)/20)*log2(((20-x)/20))};
            \addplot[mark=o, only marks,black,thick,domain=1:5] {-(x/20)*log2(x/20)-((20-x)/20)*log2(((20-x)/20))};
            \addlegendentry{$\min \ I_\lambda$};
                
            \addplot [domain=5.5:7.5,dashed,black]{1.4};
                
            \addplot[mark=x,ycomb,black,thick,domain=8:12] {-(x/20)*log2(x/20)-((20-x)/20)*log2((1/20))};
            \addplot[mark=o,ycomb,white,very thick,domain=8:12] {-(x/20)*log2(x/20)-((20-x)/20)*log2(((20-x)/20))};
            \addplot[mark=o, only marks,black,thick,domain=8:12] {-(x/20)*log2(x/20)-((20-x)/20)*log2(((20-x)/20))};
                
            \addplot [domain=12.5:14.5,dashed,black]{0.825};
                
            \addplot[mark=x,ycomb,black,thick,domain=15:19] {-(x/20)*log2(x/20)-((20-x)/20)*log2((1/20))};
			\addplot[mark=o,ycomb,white,very thick,domain=15:19] {-(x/20)*log2(x/20)-((20-x)/20)*log2(((20-x)/20))};
			\addplot[mark=o, only marks,black,thick,domain=15:19] {-(x/20)*log2(x/20)-((20-x)/20)*log2(((20-x)/20))};
    \end{axis}
\end{tikzpicture}
    \caption{Discrete plot of the range of $I_\lambda$ for fixed $n$ and varying $|CV_\lambda|$, between the Minimum--Maximum bounds in Proposition \ref{minmaxBdds}. For $n$ even (as shown here), the maximum lower--bound occurs at $n/2$.}
    \label{fig:entropyPlt}
\end{figure}

\begin{remark} \label{structComm}
    There are a number of observations that one can make on the range of values that $I_\lambda$ may take as the number of $\lambda$--core vertices changes and $n$ remains fixed. Firstly, the lower--bound is attained when the number of blocks in $\mathcal{V}_\lambda$ is exactly two, while the upper--bound is attained when the number of blocks is $n - |CV_\lambda| + 1$. Hence for fixed $n$ and $|CV_\lambda|$, the entropy $I_\lambda$ increases as $D_\lambda$ increases. Secondly, the maximum lower--bound is attained when $\mathcal{V}_\lambda$ has two blocks of more or less equal size. Entropy values below this threshold indicate graphs with either very few or almost all vertices being $\lambda$--core. Moreover, $D_\lambda$ must be very small and hence few vertices must be distant from a $\lambda$--core vertex.
\end{remark}

\section{Applications}

For every eigenvalue $\lambda$ of a universal adjacency matrix $\mathbf{U}$ associated with a graph $G$, a notion of information content related to that eigenvalue may be defined, based on the structure of the $\lambda$--CDP. We say that a graph $G$ is \textit{singular} if its adjacency matrix is singular. Partitions of the vertex set based on neighbourhoods of $\lambda$--core vertices have been studied greatly in \cite{ScMfBg, JauMol} for singular graphs with independent $CV_\lambda$. Such partitions, as well as topological indices, have many applications in the study of networks and molecular graphs.

\subsection{Singular Graphs}

In the study of singular graphs, \textit{minimal configurations} (MCs) are of importance, as they are considered to be the building blocks for such graphs, through graph operations such as coalescence. 

\begin{definition}[\cite{SciNullityOne}]
A minimal configuration is a singular graph which is either $K_1$ or if $n \geq 3$ then it has a core--graph $F$ induced by $CV_0$ and a periphery $\mathcal{P}\backslash CV_0$ satisfying the following conditions:
\begin{enumerate}[(i)]
    \item $\eta(G) = 1$,
    \item $\mathcal{P} = \emptyset$ or the graph induced by $\mathcal{P}$ consists of isolated vertices,
    \item $|\mathcal{P}| + 1 = \eta(F)$.
\end{enumerate}
\end{definition}
In particular, MCs are connected (Theorem 2, \cite{SciNullityOne}). Of interest are bipartite MCs, which are closely associated with \textit{slim graphs}. Slim graphs were defined for the case of independent $0$--core vertices in Definition 8 of \cite{ScMfBg}. The following definition generalises the notion of slim graphs to include those graphs for which $CV_0$ is not an independent set.
\begin{definition}
	A connected singular graph $G$ with $0$--core distance partition $\mathcal{V}_0$ for $\mathbf{U} = \mathbf{A}$, is a \textit{slim graph} if $CFV_0$ is precisely $V_{0, 1}$.
\end{definition}
In the case that $CV_0$ is independent, then $CFV_0$ is exactly $N(CV_0)$ for a slim graph. The following proposition characterises such graphs as being those graphs for which the entropic measure $I_0$ is at a minimum, for a given number of vertices $n$ and $\lambda$--core vertices $|CV_0|$.

\begin{proposition} \label{slimEntMin}
    Let $G$ be a connected singular graph. Then, $G$ is a slim graph if and only if $I_0$ is equal to the lower--bound for the given $n$ and $|CV_0|$.
\end{proposition}

\begin{proof}
    If $G$ is a slim graph, then by definition, for $\mathbf{U} = \mathbf{A}$, we have that $\mathcal{V}_0 = \{V_{0,0} = CV_0, V_{0, 1}\}$, where $|V_{0, 1}| = n - |CV_0|$. Therefore by Proposition \ref{minmaxBdds}, $I_0$ is equal to the lower--bound for the given $n$ and $|CV_0|$. Conversely, let $I_0$ be equal to the lower--bound, given $n$ and $|CV_0|$. Suppose that the $0$--CDP for $\mathbf{U} = \mathbf{A}$ has at least three blocks. Then $I_0$ has the term $-\frac{|CV_0|}{n}\log\left(\frac{|CV_0|}{n}\right)$, and all the other terms are of the form $-\frac{n-k}{n}\log\left(\frac{n-k}{n}\right)$, where $|CV_0| < k < n$. Since $-\log(x)$ is strictly monotonically decreasing, then $-\log\left(\frac{n-|CV_0|}{n}\right) < -\log\left(\frac{n-k}{n}\right)$ and hence $I_0$ is greater than the lower--bound, a contradiction. The result follows.
\end{proof}

It was shown by \cite{ScMfBg} that bipartite MCs, which have an independent $0$--core vertex set, are in fact slim graphs with partite sets $V_1 = CV_0$ and $V_2 = N(CV_0)$, such that $|CV_0| = |N(CV_0)| + 1$. The following is an immediate consequence of Proposition \ref{slimEntMin}.

\begin{proposition}
    If $G$ is a bipartite minimal configuration then $I_0$ is equal to the lower--bound for the given $n$ and $|CV_0|$.
\end{proposition}

\subsection{Chemical Considerations}

Let $G$ be the $n$--atomic molecular graph of a conjugated hydrocarbon system. In \textit{H\"uckel Molecular Orbital Theory}, the \textit{Hamiltonian} for such a system is given by $\mathbf{H}(\alpha, \beta) = \alpha \mathbf{I} + \beta \mathbf{A}$, where $\mathbf{A}$ is the adjacency matrix of $G$ and $\alpha, \beta$ are two parameters. Notice that $\mathbf{H}$ may be represented by a universal adjacency matrix, by setting $\gamma_A = \beta$, $\gamma_I = \alpha$, and $\gamma_D = 0 = \gamma_J$. Clearly, $\mathbf{H}(\alpha, \beta)$ and $\mathbf{A}$ commute and hence have the same eigenvectors. If $\lambda_i$ is an eigenvalue in the ordered spectrum of $\mathbf{A}$, then $E_i = \alpha + \beta\lambda_i$ is the corresponding eigenvalue of $\mathbf{H}(\alpha, \beta)$.

Within the H\"uckel framework, $E_i$ corresponds to the $i^{\rm th}$ energy level. Each eigenvector corresponds to one of the $n$--molecular orbitals, where the eigenvector entries give a linear combination of the $n$--atomic orbitals. The normalised values of these entries represent the probability of electron occupation at each atomic orbital, for the $i^{\rm th}$ molecular orbital. For an in--depth treatment of the H\"uckel framework, see \cite{GutPol}. 

Since $\mathbf{H}(\alpha, \beta)$ takes the form of a universal adjacency matrix, then the $E_i$-CDP is defined for every $E_i \in {\rm spec}(\mathbf{H})$. Hence for the $i^{\rm th}$ energy level, we can associate a topological index $I_{E_i}$. Such an index $I_{E_i}$ has appealing interpretations, related to the probabilistic distribution of electrons amongst the atomic orbitals. For example, for large values of $I_{E_i}$, the number of atomic orbitals with non--zero probability of electron occupation is very small, with a relatively large probability of occupation for each. Moreover the corresponding vertices in the molecular graph are very `localised', as there exists a substantial number of $E_i$--core--forbidden vertices which are `distant' from these $E_i$--core vertices. Further interpretations are discussed in Remark \ref{structComm}, in a more general context.

For simplicity, we shall consider $\mathbf{H} = \mathbf{A}$. Singular graphs have an important role in the H\"uckel framework, as the zero energy level $E = 0$, and its multiplicity, correspond to the non--bonding orbitals of the molecule. The $0$--core vertices correspond to the atomic orbitals of which the molecular orbitals for $E = 0$ are a linear combination. Consider the $0$--CDP for $\mathbf{U} = \mathbf{A}$. By Proposition \ref{slimEntMin}, $I_0$ attains a minimum if and only if $G$ is a slim graph. In particular bipartite MCs, which are a sub--class of MCs and the building blocks of singular graphs, always have $I_0$ at a minimum.

\begin{example}
More generally, suppose we have two slim graphs $H_1 = G_1 + v_1$ and $H_2 = G_2 + v_2$, where $v_1$ and $v_2$ are not cut vertices and are both $0$--core vertices. Let $H_1$ have $n_1$ vertices and $k_1 = |CV_0 (H_1)|$. Similarly for $G_2 + v_2$. \cite{AliBaptSci} showed that the coalescence of $H_1$ and $H_2$ on $v_1$ and $v_2$ yields the graph $H_1 \circ H_2 = G_1 + G_2 + v$, where $v$ is also a $0$--core vertex, and the nullity is $\eta(H_1) + \eta(H_2) - 1$. Clearly the $0$--core vertices in $H_1$ which are not $v_1$ are preserved, and similarly for $H_2$. Hence $|CV_0(H_1 \circ H_2)| = k_1 + k_2 - 1$. It is also worth noting that $H_1 \circ H_2$ is also a slim graph and therefore $I_0(H_1 \circ H_2)$ is at a minimum. Without loss of generality, let $\frac{k_1}{n_1} < \frac{k_2 - 1}{n_2 -1}$. Then $\frac{k_1}{n_1} < \frac{k_1 + k_2 - 1}{n_1 + n_2 -1} <\frac{k_2}{n_2}$. It follows that $I_0(H_1 \circ H_2)$ is some value between $I_0(H_1)$ and $I_0(H_2)$, reflecting the `averaged' distribution of $0$--core vertices from $H_1$ and $H_2$ in $H_1 \circ H_2$, both in probabilistic terms for electron occupation, and in terms of the distribution of $0$--core vertices within the molecular graph.
\end{example}

Future directions include the study of changes in $I_0$ as larger molecular graphs are constructed, through \textit{eg.} edge additions or graph coalescence operations. 

\section*{Acknowledgement}
The author is grateful to the anonymous referee for their comments, which have led to an overall improvement in the presentation of this paper.

\bibliography{lambda_dist_domn_BIB}

\end{document}